\documentclass{amsart}
\usepackage{hyperref}
\usepackage{amsmath,amssymb,amscd,mathtools}
\usepackage{stmaryrd}

\DeclareMathOperator{\Gr}{Gr}
\DeclareMathOperator{\Fl}{Fl}
\DeclareMathOperator{\HH}{H}
\DeclareMathOperator{\K}{K}
\DeclareMathOperator{\QH}{QH}
\DeclareMathOperator{\QK}{QK}
\newcommand{\BQH}{\QH^{\mathrm{\scriptstyle big}}}
\newcommand{\BQK}{\QK^{\mathrm{\scriptstyle big}}}
\DeclareMathOperator{\GW}{GW}

\DeclareMathOperator{\SL}{SL}

\DeclareMathOperator{\Sp}{Sp}

\DeclareMathOperator{\pt}{pt}
\DeclareMathOperator{\ev}{ev}

\DeclareMathOperator{\codim}{codim}
\DeclarePairedDelimiter{\angles}{\langle}{\rangle}
\DeclarePairedDelimiter{\qkpair}{(\!(}{)\!)}

\newcommand{\ssm}{\smallsetminus}
\newcommand{\bP}{{\mathbb P}}
\newcommand{\C}{{\mathbb C}}

\newcommand{\Q}{{\mathbb Q}}
\newcommand{\Z}{{\mathbb Z}}
\newcommand{\N}{{\mathbb N}}
\newcommand{\GQt}{\Gamma\llbracket Q, t \rrbracket}

\newcommand{\cF}{{\mathcal F}}
\newcommand{\cG}{{\mathcal G}}

\newcommand{\cO}{{\mathcal O}}

\newcommand{\al}{{\alpha}}
\newcommand{\be}{{\beta}}
\newcommand{\ga}{{\gamma}}

\newcommand{\wb}{\overline}
\newcommand{\ov}{\overline}
\newcommand{\PP}{{\mathbb P}}

\newcommand{\hgw}[3]{\angles{ #1 }^{#2}_{#3}}
\newcommand{\kgw}[3]{\angles{ #1 }^{#2}_{#3}}
\newcommand{\Mb}{\wb{\mathcal M}}
\newcommand{\euler}[1]{\chi_{_{#1}}}
\newcommand{\ignore}[1]{}

\newcommand{\REF}[1]{\cite{#1}}
\newcommand{\REFS}[1]{\cite{#1}}
\newcommand{\citen}[1]{\cite{#1}}

\newtheorem{theorem}{Theorem}[section]

\newtheorem{corollary}[theorem]{Corollary}

\theoremstyle{definition}
\newtheorem{remark}[theorem]{Remark}
\newtheorem{example}[theorem]{Example}

\begin{document}

\title{$\K$-theoretic Gromov--Witten invariants of line degrees on flag varieties}

\date{April 24, 2024}


\author{Anders S.~Buch}
\address{
  Department of Mathematics, Rutgers University,
  110 Frelinghuysen Road\\
  Piscataway, NJ 08854, USA\\
  asbuch@math.rutgers.edu}

\author{Linda Chen}
\address{
  Department of Mathematics and Statistics, Swarthmore College,
  500 College Avenue\\
  Swarthmore, PA 19081, USA\\
  lchen@swarthmore.edu}

\author{Weihong Xu}
\address{
  Department of Mathematics, Virginia Tech, 460 McBryde Hall,
  225 Stanger Street\\
  Blacksburg, VA 24061, USA\\
  weihong@vt.edu}

\begin{abstract}

A homology class $d \in \HH_2(X,\Z)$ of a complex flag variety $X = G/P$ is
called a \emph{line degree} if the moduli space $\Mb_{0,0}(X,d)$ of 0-pointed
stable maps to $X$ of degree $d$ is also a flag variety $G/P'$. We prove a
\emph{quantum equals classical} formula stating that any $n$-pointed
(equivariant, $\K$-theoretic, genus zero) Gromov--Witten invariant of line
degree on $X$ is equal to a classical intersection number computed on the flag
variety $G/P'$. We also prove an $n$-pointed analogue of the Peterson comparison
formula stating that these invariants coincide with Gromov--Witten invariants of
the variety of complete flags $G/B$. Our formulas make it straightforward to
compute the big quantum $\K$-theory ring $\BQK(X)$ modulo the ideal
$\angles{Q^d}$ generated by degrees $d$ larger than line degrees.

\end{abstract}


\maketitle

\markboth{ANDERS S.~BUCH, LINDA CHEN, AND WEIHONG XU}
{GROMOV--WITTEN INVARIANTS OF LINE DEGREES ON FLAG VARIETIES}


\section{Introduction}

In this paper we study the $n$-pointed genus zero Gromov--Witten invariants of
line degrees on a complex flag variety $X = G/P$. Given an effective degree $d
\in \HH_2(X,\Z)$, we let $\Mb_{0,n}(X,d)$ be the Kontsevich moduli space of
$n$-pointed stable maps to $X$ of degree $d$ and genus zero. Given subvarieties
$\Omega_1, \dots, \Omega_n \subset X$ in general position, the cohomological
\emph{Gromov--Witten invariant} $\hgw{[\Omega_1], \dots, [\Omega_n]}{X}{d}$
counts the number of parametrized curves $\bP^1 \to X$ of degree $d$ with $n$
marked points on the domain (up to projective transformation), such that the
$i$-th marked point is sent to $\Omega_i$ for $1 \leq i \leq n$, assuming that
finitely many such curves exist. More generally, the $\K$-theoretic
Gromov--Witten invariant $\kgw{[\cO_{\Omega_1}], \dots, [\cO_{\Omega_n}]}{X}{d}$
is defined as the sheaf Euler characteristic $\chi(\GW_d,\cO_{\GW_d})$ of the
Gromov--Witten subvariety $\GW_d \subset \Mb_{0,n}(X,d)$ of stable maps that
send the $i$-th marked point to $\Omega_i$.

The degree $d \in \HH_2(X,\Z)$ is called a \emph{line degree} if $G$ acts
transitively on the moduli space $\Mb_{0,0}(X,d)$ of zero-pointed stable maps of
degree $d$. This happens when $d = [X_{s_\al}]$ is the class of a
one-dimensional Schubert variety such that the defining simple root $\al$
satisfies combinatorial conditions given in \REFS{cohen.cooperstein:lie,
landsberg.manivel:projective*1, strickland:lines} and discussed in
Section~\ref{sec:line-deg}. All one-dimensional Schubert classes are line
degrees if $G$ is simply-laced, if $X$ is the variety $G/B$ of complete flags,
or if $X$ is a cominuscule flag variety.

When $d$ is a line degree, the 0 and 1-pointed moduli spaces $M_0 =
\Mb_{0,0}(X,d)$ and $M_1 = \Mb_{0,1}(X,d)$ are flag varieties, $M_0 = G/P'$ and
$M_1 = G/(P \cap P')$, and the natural projections $p : M_1 \to X$ and $q : M_1
\to M_0$ coincide with the evaluation map and the forgetful map. Our main result
(Theorem~\ref{thm:KGW}) is a \emph{quantum equals classical} formula stating
that, for arbitrary (equivariant) $\K$-theory classes $\cF_1, \dots, \cF_n \in
\K_T(X)$, the associated Gromov--Witten invariant of line degree $d$ is given by
\begin{equation}\label{eqn:main}%
  \kgw{\cF_1,\dots,\cF_n}{X}{d} \ = \
  \euler{M_0}\left( q_* p^* \cF_1 \cdot \ldots \cdot q_* p^* \cF_n \right) \,,
\end{equation}
where $\euler{M_0} : \K_T(M_0) \to \K_T(\pt)$ is the sheaf Euler characteristic
map. The proof uses that $\Mb_{0,n}(X,d)$ has rational singularities and is
birational to the $n$-fold product of $M_1$ over $M_0$. We also prove an
$n$-pointed analogue of the Peterson comparison formula, stating that the
Gromov--Witten invariant (\ref{eqn:main}) coincides with a Gromov--Witten
invariant of the variety of complete flags $G/B$.

Early results in this direction were proved for 3-pointed Gromov--Witten
invariants of classical Grassmannians and cominuscule flag varieties
\cite{buch:quantum, buch.kresch.ea:gromov-witten, buch.kresch.ea:quantum,
chaput.manivel.ea:quantum*1, leung.li:classical, buch.mihalcea:quantum,
chaput.perrin:rationality}. The comparison formula and a variant of
(\ref{eqn:main}) were proved for 3-pointed Gromov--Witten invariants in
\REF{li.mihalcea:k-theoretic}. Our results imply an analogous formula for
$n$-pointed cohomological Gromov--Witten invariants, special cases of which were
obtained in \REFS{chen.gibney.ea:equivalence} and \citen{perrin.smirnov:big}.

The (equivariant) big quantum \(\K\)-theory ring \(\BQK(X)\) introduced in
\REFS{givental:wdvv, lee:quantum*1} is a power series deformation of the
\(\K\)-theory ring $\K(X)$ that encodes the $n$-pointed $\K$-theoretic
Gromov--Witten invariants of all degrees. Degrees of curves are encoded as
powers $Q^d$ of Novikov variables, and additional variables $t_w$ dual to
Schubert classes $\cO^w \in \K(X)$ encode insertions in Gromov--Witten
invariants. The small quantum $\K$-theory ring \(\QK(X)\) is recovered when all
\(t_w\) are specialized to \(0\).

Our formulas make it straightforward to compute the multiplicative structure of
$\BQK(X)$ modulo the ideal $\angles{Q^e}$ generated by degrees $e$ larger than
line degrees. We provide some examples of products in this ring in the last
section. There are other approaches to computing $n$-pointed $\K$-theoretic
Gromov--Witten invariants when \(\K(X)\) is multiplicatively generated by line
bundles; for example, some computations in type $A$ have been obtained by using
the $J$-function \cite{givental.lee:quantum, lee.pandharipande:reconstruction,
iritani.milanov.ea:reconstruction}.


\section{Preliminaries}
\label{sec:lines}

\subsection{Flag varieties}

Let $G$ be a complex reductive linear algebraic group, and fix a maximal torus
$T$, a Borel subgroup $B$, and a parabolic subgroup $P$, such that $T \subset B
\subset P \subset G$. The opposite Borel subgroup $B^- \subset G$ is defined by
$B^- \cap B = T$. Let $\Phi$ be the associated root system, with basis of simple
roots $\Delta \subset \Phi^+$. Let $W = N_G(T)/T$ be the Weyl group of $G$, and
$W_P = N_P(T)/T$ the Weyl group of $P$. Then $W$ is generated by the simple
reflections $\{ s_\al \mid \al \in \Delta \}$, and $W_P$ is determined by (and
determines) the set $\Delta_P = \{ \al \in \Delta \mid s_\al \in W_P \}$.

Let $X = G/P$ be the flag variety defined by $P$. Any Weyl group element $w \in
W$ defines the Schubert varieties $X_w = \ov{B w P/P}$ and $X^w = \ov{B^- w
P/P}$ in $X$. These varieties depend only on the coset $w W_P$ in $W/W_P$. When
$w$ belongs to the subset $W^P \subset W$ of minimal representatives of the
cosets in $W/W_P$, we have $\dim(X_w) = \codim(X^w,X) = \ell(w)$, where
$\ell(w)$ denotes the Coxeter length of $w$.

\subsection{Gromov--Witten invariants}

For any projective $T$-variety $Y$, let $\K_T(Y)$ be the equivariant $\K$-theory
ring, defined as the Grothendieck ring of $T$-equivariant algebraic vector
bundles. This ring is an algebra over $\K_T(\pt)$, the representation ring of
$T$. Let $\euler{Y} : \K_T(Y) \to \K_T(\pt)$ be the push-forward map along the
structure morphism. The equivariant $\K$-theory $\K_T(X)$ of the flag variety $X
= G/P$ is a free $\K_T(\pt)$-module with basis $\{ \cO^w \mid w \in W^P \}$,
where $\cO^w = [\cO_{X^w}] \in \K_T(X)$ is the $\K$-theoretic Schubert class
defined by the structure sheaf of $X^w$.

The homology group $\HH_2(X,\Z)$ is a free abelian group generated by the
classes $[X_{s_\al}]$ of the $1$-dimensional Schubert varieties for $\al \in
\Delta \ssm \Delta_P$. Given an effective degree $d \in \HH_2(X,\Z)$ and $n \in
\N$, we let $\Mb_{0,n}(X,d)$ denote the Kontsevich moduli space of $n$-pointed
stable maps to $X$ of degree $d$ and genus zero (see
\REF{fulton.pandharipande:notes}). This moduli space is non-empty when $d \neq
0$ or $n \geq 3$. In this case, the evaluation map $\ev_i : \Mb_{0,n}(X,d) \to
X$, defined for $1 \leq i \leq n$, sends a stable map to the image of the $i$-th
marked point in its domain.

Given $\K$-theory classes $\cF_1, \dots, \cF_n \in \K_T(X)$, the corresponding
$n$-pointed (equivariant) $K$-theoretic Gromov--Witten invariant of degree $d$
and genus zero is defined by
\[
  \kgw{\cF_1, \dots, \cF_n}{X}{d} \ = \
  \euler{\Mb_{0,n}(X,d)}\left( \prod_{i=1}^n \ev_i^* \cF_i \right)
  \ \in \K_T(\pt) \,.
\]
Similarly, the cohomological Gromov--Witten invariant given by $\ga_1, \dots,
\ga_n \in \HH^*_T(X)$ is defined by
\[
  \hgw{\ga_1, \dots, \ga_n}{X}{d} \ = \
  \int_{\Mb_{0,n}(X,d)}\, \prod_{i=1}^n \ev_i^*(\ga_i) \ \in \HH_T(\pt) \,.
\]
Non-equivariant Gromov--Witten invariants are obtained by replacing $T$ with the
trivial group; these Gromov--Witten invariants are integers.


\section{Gromov--Witten invariants of line degrees}\label{sec:line-deg}

A non-zero homology class $d \in \HH_2(X,\Z)$ will be called a \emph{line
degree} if $G$ acts transitively on the moduli space $\Mb_{0,0}(X,d)$ of
0-pointed stable maps to $X$ of degree $d$ and genus zero. Equivalently, $d =
[X_{s_\al}]$ is the class of a one-dimensional Schubert variety defined by a
simple root $\al \in \Delta\ssm\Delta_P$ such that $\al$ is a long root within
its connected component of $\Delta_P \cup \{\al\}$ \cite{cohen.cooperstein:lie,
landsberg.manivel:projective*1, strickland:lines}. Here we identify the simple
roots $\Delta$ with the set of nodes in the Dynkin diagram of $\Phi$, so that
the connected component of $\al$ in $\Delta_P \cup \{\al\}$ is an irreducible
Dynkin diagram in itself. We further use the convention that all roots of a
simply-laced root system are long. In particular, $[X_{s_\al}]$ is a line degree
if the component of $\al$ in $\Delta_P \cup \{\al\}$ is simply-laced, even if
$\al$ is a short root of $\Phi$. All one-dimensional Schubert classes are line
degrees if $\Phi$ is simply-laced, if $X$ is the variety $G/B$ of complete
flags, or if $X$ is a cominuscule flag variety. We note that the definition of
line degree depends on the group $G$. For example, the projective space
$\bP^{2n-1}$ is a flag variety of both $\SL(2n)$ and $\Sp(2n)$, and $\SL(2n)$
acts transitively on the set of lines in $\bP^{2n-1}$ whereas $\Sp(2n)$ does
not.

Given a fixed line degree $d = [X_{s_\al}]$, we let $P' \subset G$ be the
parabolic subgroup defined by
\[
  \Delta_{P'} = (\Delta_P \cup \{\al\}) \ssm
  \{ \be \in \Delta \mid (\be, \al^\vee) < 0 \} \,,
\]
that is, $\Delta_{P'}$ is obtained from $\Delta_P \cup \{\al\}$ by removing the
simple roots adjacent to $\al$ in the Dynkin diagram. In this case the moduli
spaces of 0 and 1-pointed stable maps to $X$ of degree $d$ and genus zero are
the flag varieties $M_0 = \Mb_{0,0}(X,d) = G/P'$ and $M_1 = \Mb_{0,1}(X,d) =
G/(P \cap P')$, and the natural projections $p : G/(P \cap P') \to X$ and $q :
G/(P \cap P') \to G/P'$ coincide with the evaluation map and the forgetful
map \cite{cohen.cooperstein:lie, landsberg.manivel:projective*1,
strickland:lines}.
\[
  \begin{CD}
    M_1 = G/(P\cap P') @>\ev \,=\, p >> X=G/P \\@V q VV \\ M_0 = G/P'
  \end{CD}
\]
The curve of degree $d$ in $X$ corresponding to $y \in M_0$ is given by
\[
  L_y \,=\, p\left(q^{-1}(y)\right) \,.
\]

Let $\pi_X : G/B \to X$ be the projection. If $d = [X_{s_\al}]$ is a line degree
of $X$, we also let $d$ denote the unique line degree $[B s_\al B/B]$ of $G/B$
that is mapped to $d$ by push-forward along $\pi_X$. This class $[B s_\al B/B]$
is the \emph{Peterson lift} of $d$, see Definition~1 in \REF{woodward:d}. In
fact, one can check that any class in $\HH_2(X;\Z)$ is a line degree of $X$ if
and only if its Peterson lift is a line degree of $G/B$.

Our main result is the following theorem. Part (B) and a variant of (A) were
proved for 3-pointed Gromov--Witten invariants in \REF{li.mihalcea:k-theoretic}.

\begin{theorem}\label{thm:KGW}%
  Let $d \in \HH_2(X,\Z)$ be a line degree of the flag variety $X$ with
  associated projections $p : M_1 \to X$ and $q : M_1 \to M_0$, and let $\cF_1,
  \dots, \cF_n \in \K_T(X)$ be $\K$-theory classes. The following identities
  hold in $\K_T(\pt)$:\smallskip

  \noindent{\emph{(A)}} \ $\kgw{\cF_1, \dots, \cF_n}{X}{d} \, = \,
  \euler{G/P'}(q_* p^* \cF_1 \cdot \ldots \cdot q_* p^* \cF_n)$.
  \smallskip

  \noindent{\emph{(B)}} \ $\kgw{\cF_1, \dots, \cF_n}{X}{d} \, = \,
  \kgw{\pi_X^*\, \cF_1, \dots, \pi_X^*\, \cF_n}{G/B}{d}$ \,.
\end{theorem}
\begin{proof}
  To prove part (A), let $M_1^{(n)} = M_1\times_{M_0} \dots \times_{M_0} M_1$ be
  the fiber product of $n$ copies of $M_1$ over $M_0$, with projections $e_i :
  M_1^{(n)} \to M_1$ for $1 \leq i \leq n$. Set $M_n = \Mb_{0,n}(X,d)$, and let
  $\phi : M_n \to M_1^{(n)}$ be the morphism defined by the $n$ forgetful maps
  $M_n \to M_1$. We obtain the commutative diagram:
  \[
    \begin{CD}
      M_1^{(n)} @>e_i>> M_1 @>q>> M_0 \\
      @A \phi AA @V p VV \\
      M_n @>\ev_i>> X
    \end{CD}
  \]
  Since $q$ is a locally trivial fibration with non-singular base and fiber, it
  follows that $M^{(n)}_1$ is a non-singular projective variety. Using that any
  morphism $\bP^1 \to X$ of line degree is an isomorphism onto its image, it
  follows that $\phi$ is birational. The variety $M_n$ has rational
  singularities by Theorem~2~(ii) in \REF{fulton.pandharipande:notes} and
  Proposition~5.15 in \REF{kollar.mori:birational}. We obtain $\phi_*
  [\cO_{M_n}] = [\cO_{M_1^{(n)}}]$ in $\K_T(M_1^{(n)})$, and therefore
  \[
  \newcommand{\ts}{\textstyle}
    \begin{split}
      \kgw{\cF_1, \dots, \cF_n}{X}{d} \
      &= \
      \euler{M_n}\left( \ts\prod_{i=1}^n \ev_i^* \cF_i \right)
      \ = \
      \euler{M_n}\left( \ts \phi^* \prod_{i=1}^n e_i^*\, p^* \cF_i \right) \\
      &= \
      \euler{M_1^{(n)}}\left( \ts\prod_{i=1}^n e_i^*\, p^* \cF_i \right)
      \ = \
      \euler{M_0}\left( \ts\prod_{i=1}^n q_* p^* \cF_i \right) \,,
    \end{split}
  \]
  where the last two equalities follow from the projection formula and Lemma 3.5
  in \REF{buch.mihalcea:quantum}.

  For part (B), let $P_\al \subset G$ be the minimal parabolic subgroup given by
  $\Delta_{P_\al} = \{\al\}$, where $\al \in \Delta\ssm\Delta_P$ is defined by
  $d = [X_{s_\al}]$. We obtain a commutative diagram
  \[
    \begin{CD}
      G/B @>\pi>> M_1 @>p>> X \\
      @V\pi_\al VV @V q VV \\
      G/P_\al @> q_\al >> M_0
    \end{CD}
  \]
  where $\pi$, $\pi_\al$, and $q_\al$ are the natural projections of flag
  varieties. Using that the square is Cartesian and $\pi_X = p \circ \pi$, we
  obtain
  \[
    \newcommand{\ts}{\textstyle}
    \begin{split}
      \kgw{\cF_1,\dots,\cF_n}{X}{d} \
      &= \
      \euler{M_0} \left( \ts\prod_{i=1}^n q_* p^* \cF_i \right)
      \ = \
      \euler{G/P_\alpha}
      \left( \ts\prod_{i=1}^n q_\alpha^*\, q_* p^* \cF_i \right) \\
      &= \
      \euler{G/P_\alpha}
      \left( {\ts\prod_{i=1}^n \pi_{\alpha *} \pi^* p^* \cF_i} \right)
      \ =
      \
      \euler{G/P_\alpha}
      \left( \ts\prod_{i=1}^n \pi_{\alpha *} \pi_X^* \cF_i \right)
      \\
      &= \
      \kgw{\pi_X^* \cF_1, \dots, \pi_X^* \cF_n}{G/B}{d} \,,
    \end{split}
  \]
  where the last equality follows from part (A) applied to $G/B$, noting that we
  have $\Mb_{0,0}(G/B,d) = G/P_\al$ and $\Mb_{0,1}(G/B,d) = G/B$.
\end{proof}

Theorem~\ref{thm:KGW} implies the analogous identities for cohomological
Gromov--Witten invariants in the following corollary by the equivariant
Hirzebruch formula \cite{edidin.graham:riemann-roch}, see section 4.1 in
\REF{buch.mihalcea:quantum}. A similar statement is proved in
\REF{perrin.smirnov:big} when $P$ is maximal and used to study semisimplicity of
big quantum cohomology; when $X$ is a Grassmannian, the statement is proved and
applied in \REF{chen.gibney.ea:equivalence}. Part (B) for 3-pointed
Gromov--Witten invariants is a special case of Peterson's comparison formula,
proved in Woodward's paper \REF{woodward:d}.

\begin{corollary}\label{cor:GW}%
  Let $d \in \HH_2(X,\Z)$ be a line degree of $X$, and let $\ga_1, \dots, \ga_n
  \in \HH^*_T(X)$ be cohomology classes. The following identities hold in
  $\HH^*_T(\pt)$:\smallskip

  \noindent{\emph{(A)}} \ $\hgw{\ga_1, \dots, \ga_n}{X}{d} \ = \
  \int_{G/P'}\, q_* p^* \ga_1 \cdot \ldots \cdot q_* p^* \ga_n$.
  \smallskip

  \noindent{\emph{(B)}} \ $\hgw{\ga_1, \dots, \ga_n}{X}{d} \ = \
  \hgw{\pi_X^* \ga_1, \dots, \pi_X^* \ga_n}{G/B}{d}$.
\end{corollary}

\begin{remark}\label{rmk:enum}%
  Theorem~\ref{thm:KGW} and Corollary~\ref{cor:GW} imply that non-equivariant
  Gromov--Witten invariants of line degree are enumerative for Schubert classes
  in the following sense. Let $\Omega_1, \dots, \Omega_n \subset X$ be Schubert
  varieties in general position and let
  \[
    Y \ = \ \bigcap_{i=1}^n q \left(p^{-1}(\Omega_i)\right)
    \ = \ \{ y \in M_0 \mid L_y \cap \Omega_i \neq \emptyset ~\,\forall\,
    1 \leq i \leq n \}
  \]
  be the subvariety of $M_0 = G/P'$ parametrizing lines meeting $\Omega_1,
  \dots, \Omega_n$. Then
  \begin{equation}\label{enum:KGW}%
    \kgw{[\cO_{\Omega_1}],\dots,[\cO_{\Omega_n}]}{X}{d} \
    = \ \chi(Y, \cO_Y) \,.
  \end{equation}
  This follows from Theorem~\ref{thm:KGW}~(A) since the projection $q :
  p^{-1}(\Omega_i) \to q(p^{-1}(\Omega_i))$ of the Schubert variety
  $p^{-1}(\Omega_i)$ is cohomologically trivial \cite{ramanathan:schubert}. The
  left hand side of (\ref{enum:KGW}) is the sheaf Euler characteristic
  $\chi(\GW_d,\cO_{\GW_d})$ of the Gromov--Witten variety $\GW_d =
  \bigcap_{i=1}^n \ev_i^{-1}(\Omega_i) \subset \Mb_{0,n}(X,d)$ of stable maps
  that send the $i$-th marked point to $\Omega_i$ (see section 4.1 in
  \REF{buch.mihalcea:quantum}).

  When the Schubert varieties satisfy the condition
  \[
    \sum_{i=1}^n \codim(\Omega_i,X) \, = \, \dim \Mb_{0,n}(X,d) \, = \,
    \dim M_0 + n \,,
  \]
  we also have
  \begin{equation}\label{eqn:enum:GW}%
    \hgw{[\Omega_1], \dots, [\Omega_n]}{X}{d} \ = \
    \int_{G/P'} [Y] \ = \ \# Y \,,
  \end{equation}
  that is, the cohomological Gromov--Witten invariant $\hgw{[\Omega_1], \dots,
  [\Omega_n]}{X}{d}$ is the number of lines in $X$ of degree $d$ meeting
  $\Omega_1, \dots, \Omega_n$.
\end{remark}

\begin{example}
  Let $\ga \in \HH^*(\bP^3)$ be the class of a line. Then Remark~\ref{rmk:enum}
  implies that $\hgw{\ga,\ga,\ga,\ga}{\bP^3}{1}$ counts the number of lines that
  meet 4 given lines in general position in $\bP^3$, and this number can be
  computed as a classical intersection number on $M_0 = \Gr(2,4)$.
\end{example}


\section{Applications to big quantum K-theory}

\subsection{Definitions}

Set $\Gamma = \K_T(\pt) \otimes \Q$.
Given a fixed flag variety $X = G/P$, we let
\[
  \GQt \,=\, \Gamma\llbracket Q_\al, t_w \mid \al \in \Delta\ssm\Delta_P,\,
  w \in W^P \rrbracket
\]
be the ring of formal power series over $\Gamma$, in Novikov variables $Q_\al$
dual to the Schubert basis of $\HH_2(X,\Z)$, and formal variables $t_w$ dual to
the Schubert basis of $\K_T(X)$. The big equivariant quantum $K$-theory ring of
$X$ is a $\GQt$-algebra defined by
\[
  \BQK_T(X) = \K_T(X,\Q) \otimes_\Gamma \GQt
\]
as a module. We proceed to define the multiplicative structure on $\BQK_T(X)$
following \REFS{givental:wdvv, lee:quantum*1}.

For $d = \sum_{\al \in \Delta\ssm\Delta_P} d_\al [X_{s_\al}]
\in \HH_2(X,\Z)$ we write
\[
  Q^d = \prod_{\al \in \Delta\ssm\Delta_P} Q_\al^{\,d_\al} \,,
\]
and for any function $h : W^P \to \N$ we define
\[
  t^h = \prod_{w \in W^P} t_w^{\,h(w)} \quad , \quad\quad
  h! = \prod_{w \in W^P} h(w)! \quad\text{ , \ and}\quad\quad
  |h| = \sum_{w \in W^P} h(w) \,.
\]
Given $\K$-theory classes $\cF_1, \cF_2, \cF_3 \in \K_T(X)$, we let $\kgw{\cF_1,
\cF_2, \cF_3, \cO^h}{X}{d}$ denote the $(|h|+3)$-pointed Gromov--Witten
invariant of degree $d$ with $h(w)$ insertions of $\cO^w$, for $w \in W^P$, in
addition to the first three insertions. We then define
\[
  \qkpair{\cF_1, \cF_2, \cF_3}
  \ = \ \sum_{d,h} \kgw{\cF_1, \cF_2, \cF_3, \cO^h}{X}{d}\, \frac{t^h}{h!}\, Q^d
  \ \in \GQt \,,
\]
with the sum over all effective degrees $d \in \HH_2(X,\Z)$ and functions $h :
W^P \to \N$. We extend this by linearity to a symmetric 3-form on the
$\GQt$-module $\BQK_T(X)$. The \emph{quantum metric} on $\BQK_T(X)$ is then
defined by $\qkpair{\cF_1, \cF_2} = \qkpair{\cF_1, \cF_2, 1}$, and the
\emph{quantum product} $\cF_1 \star \cF_2 \in \BQK_T(X)$ is the unique class
defined by
\[
  \qkpair{\cF_1 \star \cF_2, \cF_3} = \qkpair{\cF_1, \cF_2, \cF_3}
\]
for all $\cF_3 \in K_T(X)$. The small quantum $\K$-theory ring is the quotient
$\QK_T(X) = \BQK_T(X)/\angles{t}$ by the ideal generated by $t_w$ for $w \in
W^P$.

\begin{remark}
  Let $0$ denote the identity element of $W$, so that $t_0$ is dual to $1 \in
  \K_T(X)$. The product $\cF_1 \star \cF_2 \in \BQK_T(X)$ is known to be
  independent of $t_0$ for $\cF_1, \cF_2 \in \K_T(X)$; this follows from
  \REF{givental:wdvv} (see also\ Prop.~2.10 in
  \REF{iritani.milanov.ea:reconstruction}). In fact, since the general fibers of
  the forgetful map $\Mb_{0,n+1}(X,d) \to \Mb_{0,n}(X,d)$ are isomorphic to
  $\bP^1$, we have $\kgw{1,\cO^h}{X}{d} = \kgw{\cO^h}{X}{d}$ for all effective
  $d \in \HH_2(X,\Z)$ and $h : W^P \to \N$ with $|h| \geq 3$, and therefore
  $\qkpair{\cF_1, \cF_2, \cF_3} = e^{t_0}\, \qkpair{\cF_1, \cF_2,
  \cF_3}|_{t_0=0}$ for all $\cF_1, \cF_2, \cF_3 \in \K_T(X)$. We obtain
  \[
    \begin{split}
      \qkpair{(\cF_1 \star \cF_2)|_{t_0=0}, \cF_3}
      \ &= \
      e^{t_0}\, \qkpair{(\cF_1 \star \cF_2)|_{t_0=0}, \cF_3}|_{t_0=0}
      \ = \
      e^{t_0}\, \qkpair{\cF_1 \star \cF_2, \cF_3}|_{t_0=0} \\
      &= \ e^{t_0}\, \qkpair{\cF_1, \cF_2, \cF_3}|_{t_0=0}
      \ = \ \qkpair{\cF_1, \cF_2, \cF_3} \,.
    \end{split}
  \]
  It follows that $\cF_1 \star \cF_2 \in \BQK_T(X)$ is the unique class that is
  independent of $t_0$ and satisfies
  \begin{equation}\label{eqn:bigQK0}%
    \qkpair{\cF_1 \star \cF_2, \cF_3}|_{t_0=0} \ = \
    \qkpair{\cF_1, \cF_2, \cF_3}|_{t_0=0}
  \end{equation}
  for all $\cF_3 \in \K_T(X)$. Notice that the product $\cF_1 \star \cF_2$ can
  be constructed from the \emph{quantum potential} $\cG(Q,t) = \qkpair{1,1,1}
  \in \GQt$ specialized at $t_0=0$,
  \[
    \cG_0 \, = \, \qkpair{1, 1, 1}|_{t_0=0} \ = \,
    \sum_{d,h:\,h(0)=0} \kgw{1,1,1,\cO^h}{X}{d}\, \frac{t^h}{h!}\, Q^d \
    \in \GQt \,,
  \]
  by observing that $\qkpair{\cO^u, \cO^v, \cO^w}|_{t_0=0}\, = \partial_{t_u}
  \partial_{t_v} \partial_{t_w} \cG_0$ holds for $u, v, w \in W^P \ssm \{0\}$.
\end{remark}

\begin{remark}
  The product in the equivariant big quantum cohomology ring $$\BQH_T(X) =
  \HH^*_T(X) \otimes_\Z \Q\llbracket Q,t \rrbracket$$ is defined by
  \[
    [X^u] \star [X^v] =
    \sum_{w,d,h} \hgw{[X^u], [X^v], [X_w], [X]^h}{X}{d}\,
    Q^d\, \frac{t^h}{h!}\, [X^w] \,,
  \]
  where $\hgw{[X^u], [X^v], [X_w], [X]^h}{X}{d}$ is the cohomological
  Gromov--Witten invariant with $h(w)$ insertions of $[X^w]$ in addition to the
  first three insertions.
\end{remark}

\subsection{Examples}

In this section we apply Theorem~\ref{thm:KGW} to compute some examples of big
quantum $\K$-theory products modulo powers of the Novikov variables of degrees
larger than line degrees. In each case there is a unique line degree, and only
one Novikov variable which will be denoted by $Q$. Congruence $\equiv$ is
always modulo $Q^2$. Our examples are compatible with a positivity property of
big quantum $K$-theory that we plan to discuss elsewhere.

\begin{example}
  Let $X = \bP^1 = \SL(2)/B$. The only line degree in $\HH_2(\bP^1,\Z)$ gives
  $M_1 = \bP^1$ and $M_0 = \{\pt\}$. Let $T=\C^*$ act on $\bP^1$ and let $P \in
  \bP^1$ be a $T$-fixed point. Then $\K_T(\bP^1)$ has basis $\{ 1, \cO^1 \}$,
  where $\cO^1 = [\cO_P]$. Set $a = 1-\cO^1|_P \,\in \K_T(\pt)$. Using that
  $(\cO^1)^n = (1-a)^{n-1} \cO^1 \in \K_T(\bP^1)$ for $n \geq 1$, we compute the
  specialized potential (modulo $Q^2$) as
  \[
    \cG_0 \ \equiv \ \frac{e^{(1-a)t}-a}{1-a} + Q\, e^t \,,
  \]
  where $t = t_1$ is dual to $\cO^1$. If we write $\cO^1 \star \cO^1 = c_0 + c_1
  \cO^1$ with $c_0, c_1 \in \GQt$, then the equations
  \[
    c_0 \qkpair{1, \cF} + c_1 \qkpair{\cO^1, \cF} =
    \qkpair{c_0 + c_1 \cO^1, \cF} =
    \qkpair{\cO^1\star\cO^1, \cF} \,,
  \]
  for $\cF \in \{1, \cO^1\}$, are equivalent to
  \[
    c_0\, \frac{\partial^i \cG_0}{\partial t^i} +
    c_1\, \frac{\partial^{i+1} \cG_0}{\partial t^{i+1}} \, = \,
    \frac{\partial^{i+2} \cG_0}{\partial t^{i+2}}
  \]
  for $i \in \{0, 1\}$. By solving for $c_0$ and $c_1$ modulo $Q^2$, we arrive
  at
  \[
    \cO^1\star\cO^1 \ \equiv \
    a\,Q\,e^{t} + \left( 1-a - a\,Q\,\frac{e^{t}-e^{at}}{1-a} \right) \cO^1 \,.
  \]
\end{example}

\begin{example}
  Let \(X=\PP^2\). Then \(M_1=\Fl(3)\) is a complete flag variety and
  \(M_0={\PP^2}^*\) is the dual projective plane. The Schubert basis is \(\{1,
  \cO^{1},\cO^{2}\}\), where $\cO^1 = [\cO_L]$ is the class of a line and $\cO^2
  = [\cO_P]$ is the class of a point in $\bP^2$.

  Working non-equivariantly for simplicity, we have
  \[
    \cG_0 \ \equiv \
    1+t_1+t_2+\frac{t_1^2}{2}+Q e^{t_1}\left(1+t_2+\frac{t_2^2}{2}\right) \,,
  \]
  from which we obtain (see also section~4.3 in
  \REF{iritani.milanov.ea:reconstruction}):
  \[
    \newcommand{\ts}{\textstyle}
    \begin{split}
      \cO^1 \star \cO^1 \
      &\ts\equiv \
      \cO^2 + Q e^{t_1}\Big(
      t_2 + \left( \frac{t_2^2}{2} -t_1t_2 -t_2 \right)\cO^1 +
      t_2 (t_1-t_2) \left(\frac{t_1}{2} + 1\right) \cO^2 \Big) \ ;
      \\
      \cO^1 \star \cO^2 \
      &\ts\equiv \
      Q e^{t_1}\left(1 +(t_2-t_1) \cO^1
      +\left(\frac{t_1^2}{2}-t_1t_2 - t_2 \right)\cO^2\right) \ ;
      \\
      \cO^2 \star \cO^2 \ \ts
      &\ts\equiv \
      Q e^{t_1}\left(\cO^1-t_1\cO^2\right) \,.
    \end{split}
  \]
\end{example}

\begin{example}
  Let \(X=\Gr(2,4)\) be the Grassmannian of \(2\)-planes in \(\C^4\). Here,
  \(M_0=\Fl(1,3;4)\) is a point-hyperplane incidence variety and \(M_1=\Fl(4)\)
  is a complete flag variety. The Schubert basis is \(\{1, \cO^{1}, \cO^{1,1},
  \cO^{2}, \cO^{2,1}, \cO^{2,2}\}\), where $\cO^\lambda$ is the Schubert class
  indexed by the partition $\lambda$. Working non-equivariantly, we obtain (by a
  computation in Maple):
  \[
    \textstyle
    \cO^{2,2}\star\cO^{2} \ \equiv \
    Q e^{t_1} \Big( \cO^{1,1} + \left( t_{1,1} -t_1 \right)\cO^{2,1} +
    \left( \frac{t_1^2}{2} -t_{1,1} -t_1t_{1,1} \right)\cO^{2,2} \Big)
  \]
  and
  \[
    \newcommand{\ts}{\textstyle}
    \begin{split}
    & \cO^{2}\star\cO^{2} \ \equiv \ \cO^{2,2} \, + \, Q e^{t_1} \Big( \\
    & \ts\ \
    t_{1,1}\,\cO^1 +
    \left(\frac{t_{1,1}^2}{2} +t_2t_{1,1} -t_1t_{1,1} +t_{2,1} \right) \cO^{1,1}
    +\left(\frac{t_{1,1}^2}{2} -t_{1}t_{1,1} -t_{1,1} \right) \cO^2
    \\
    & \ts\
    + \left(\frac{t_{1,1}^3}{6} + (t_2-2t_1-3)\frac{t_{1,1}^2}{2} +
      (t_1^2 -t_1t_2 +2t_1 -2t_2 +t_{2,1})t_{1,1} -t_1t_{2,1} -t_{2,1}
    \right)\cO^{2,1}
    \\
    & \ts\
    + \Big(\left( -t_1 -3 \right) \frac{t_{1,1}^3}{6}
      + \left( t_1^2 -t_1t_2 +3t_1 -3t_2 \right) \frac{t_{1,1}^2}{2} \\
    & \ts\quad\quad
      + \left( \frac{t_1^2 t_2}{2} +2t_1t_2 +t_2 -\frac{t_1^3}{3}
      -t_1^2 -t_1t_{2,1} -2t_{2,1} \right) t_{1,1}
      +\frac{t_1^2 t_{2,1}}{2} +t_1t_{2,1} \Big) \cO^{2,2} \Big) \,.
    \end{split}
  \]
\end{example}


\section*{Acknowledgments}

This project was initiated at the ICERM\footnote{Institute for Computational and
Experimental Research in Mathematics in Providence, RI.} Women in Algebraic
Geometry workshop in July 2020 and the ICERM Combinatorial Algebraic Geometry
program in Spring 2021. We thank A.~Gibney, L.~Heller, E.~Kalashnikov, H.~Larson
as well as P.~E. Chaput, L.~C. Mihalcea, and N.~Perrin for inspiring
collaborations on related projects. We also thank an anonymous referee for a
careful reading of our paper and for several helpful suggestions. AB was
partially supported by NSF Grant DMS-2152316, as well as DMS-1929284 while in
residence at ICERM during the Spring of 2021. LC was partially supported by NSF
Grant DMS-2101861.

\newcommand{\etalchar}[1]{$^{#1}$}


\begin{thebibliography}{CGH{\etalchar{+}}22}

\bibitem[BKT03]{buch.kresch.ea:gromov-witten}
A.~S. Buch, A.~Kresch, and H.~Tamvakis.
\newblock Gromov-{W}itten invariants on {G}rassmannians.
\newblock {\em J. Amer. Math. Soc.}, 16(4):901--915, 2003.

\bibitem[BKT09]{buch.kresch.ea:quantum}
A.~S. Buch, A.~Kresch, and H.~Tamvakis.
\newblock Quantum {P}ieri rules for isotropic {G}rassmannians.
\newblock {\em Invent. Math.}, 178(2):345--405, 2009.

\bibitem[BM11]{buch.mihalcea:quantum}
A.~S. Buch and L.~C. Mihalcea.
\newblock Quantum {$K$}-theory of {G}rassmannians.
\newblock {\em Duke Math. J.}, 156(3):501--538, 2011.

\bibitem[Buc03]{buch:quantum}
A.~S. Buch.
\newblock Quantum cohomology of {G}rassmannians.
\newblock {\em Compositio Math.}, 137(2):227--235, 2003.

\bibitem[CC98]{cohen.cooperstein:lie}
A.~M. Cohen and B.~N. Cooperstein.
\newblock Lie incidence systems from projective varieties.
\newblock {\em Proc. Amer. Math. Soc.}, 126(7):2095--2102, 1998.

\bibitem[CGH{\etalchar{+}}22]{chen.gibney.ea:equivalence}
L.~Chen, A.~Gibney, L.~Heller, E.~Kalashnikov, H.~Larson, and W.~Xu.
\newblock On an equivalence of divisors on $\overline{M}_{0,n}$ from
  gromov-witten theory and conformal blocks.
\newblock {\em Transformation Groups}, 2022.

\bibitem[CMP08]{chaput.manivel.ea:quantum*1}
P.-E. Chaput, L.~Manivel, and N.~Perrin.
\newblock Quantum cohomology of minuscule homogeneous spaces.
\newblock {\em Transform. Groups}, 13(1):47--89, 2008.

\bibitem[CP11]{chaput.perrin:rationality}
P.-E. Chaput and N.~Perrin.
\newblock Rationality of some {G}romov-{W}itten varieties and application to
  quantum {$K$}-theory.
\newblock {\em Commun. Contemp. Math.}, 13(1):67--90, 2011.

\bibitem[EG00]{edidin.graham:riemann-roch}
D.~Edidin and W.~Graham.
\newblock Riemann-{R}och for equivariant {C}how groups.
\newblock {\em Duke Math. J.}, 102(3):567--594, 2000.

\bibitem[FP97]{fulton.pandharipande:notes}
W.~Fulton and R.~Pandharipande.
\newblock Notes on stable maps and quantum cohomology.
\newblock In {\em Algebraic geometry---{S}anta {C}ruz 1995}, volume~62 of {\em
  Proc. Sympos. Pure Math.}, pages 45--96. Amer. Math. Soc., Providence, RI,
  1997.

\bibitem[Giv00]{givental:wdvv}
A.~Givental.
\newblock On the {WDVV} equation in quantum {$K$}-theory.
\newblock {\em Michigan Math. J.}, 48:295--304, 2000.
\newblock Dedicated to William Fulton on the occasion of his 60th birthday.

\bibitem[GL03]{givental.lee:quantum}
A.~Givental and Y.-P. Lee.
\newblock Quantum {$K$}-theory on flag manifolds, finite-difference {T}oda
  lattices and quantum groups.
\newblock {\em Invent. Math.}, 151(1):193--219, 2003.

\bibitem[IMT15]{iritani.milanov.ea:reconstruction}
H.~Iritani, T.~Milanov, and V.~Tonita.
\newblock Reconstruction and convergence in quantum {$K$}-theory via difference
  equations.
\newblock {\em Int. Math. Res. Not. IMRN}, (11):2887--2937, 2015.

\bibitem[KM98]{kollar.mori:birational}
J.~Koll{\'a}r and S.~Mori.
\newblock {\em Birational geometry of algebraic varieties}, volume 134 of {\em
  Cambridge Tracts in Mathematics}.
\newblock Cambridge University Press, Cambridge, 1998.
\newblock With the collaboration of C. H. Clemens and A. Corti, Translated from
  the 1998 Japanese original.

\bibitem[Lee04]{lee:quantum*1}
Y.-P. Lee.
\newblock Quantum {$K$}-theory. {I}. {F}oundations.
\newblock {\em Duke Math. J.}, 121(3):389--424, 2004.

\bibitem[LL12]{leung.li:classical}
N.~C. Leung and C.~Li.
\newblock Classical aspects of quantum cohomology of generalized flag
  varieties.
\newblock {\em Int. Math. Res. Not. IMRN}, (16):3706--3722, 2012.

\bibitem[LM03]{landsberg.manivel:projective*1}
J.~M. Landsberg and L.~Manivel.
\newblock On the projective geometry of rational homogeneous varieties.
\newblock {\em Comment. Math. Helv.}, 78(1):65--100, 2003.

\bibitem[LM14]{li.mihalcea:k-theoretic}
C.~Li and L.~C. Mihalcea.
\newblock K-theoretic {G}romov-{W}itten invariants of lines in homogeneous
  spaces.
\newblock {\em Int. Math. Res. Not. IMRN}, (17):4625--4664, 2014.

\bibitem[LP04]{lee.pandharipande:reconstruction}
Y.-P. Lee and R.~Pandharipande.
\newblock A reconstruction theorem in quantum cohomology and quantum
  {$K$}-theory.
\newblock {\em Amer. J. Math.}, 126(6):1367--1379, 2004.

\bibitem[PS]{perrin.smirnov:big}
N.~Perrin and M.~Smirnov.
\newblock On the big quantum cohomology of coadjoint varieties.
\newblock arXiv:2112.12436.

\bibitem[Ram85]{ramanathan:schubert}
A.~Ramanathan.
\newblock Schubert varieties are arithmetically {C}ohen-{M}acaulay.
\newblock {\em Invent. Math.}, 80(2):283--294, 1985.

\bibitem[Str02]{strickland:lines}
E.~Strickland.
\newblock Lines in {$G/P$}.
\newblock {\em Math. Z.}, 242(2):227--240, 2002.

\bibitem[Woo05]{woodward:d}
C.~T. Woodward.
\newblock On {D}. {P}eterson's comparison formula for {G}romov-{W}itten
  invariants of {$G/P$}.
\newblock {\em Proc. Amer. Math. Soc.}, 133(6):1601--1609, 2005.

\end{thebibliography}

\end{document}